\documentclass[11pt]{article}
\usepackage{geometry}
\usepackage[latin1]{inputenc}
\usepackage[italian,english]{babel}
\usepackage{amsmath}
\usepackage{amsfonts}
\usepackage{amsthm}
\usepackage{paralist, times, indentfirst}
\usepackage[T1]{fontenc}
\usepackage{fancyhdr}
\usepackage{hyperref}
\newcommand{\R}{\mathbb{R}}

\DeclareMathOperator{\sgn}{sgn}
\DeclareMathOperator{\esssup}{ess \ sup}
\newcommand{\ab}{(Amar and Bellettini, 1994)}
\newcommand{\abt}{(Alvino et al., 1990)}
\newcommand{\aftl}{(Alvino et al., 1997)}
\newcommand{\atI}{(Alvino and Trombetti, 1979)}
\newcommand{\atII}{(Alvino and Trombetti, 1978)}
\newcommand{\atl}{(Alvino et al., 1990)}
\newcommand{\bfm}{(Betta et al.,1994)}
\newcommand{\bm}{(Betta and Mercaldo, 1991)}
\newcommand{\dgI}{(Della Pietra and Gavitone)}
\newcommand{\dgII}{(Della Pietra and Gavitone, 2013)}

\newcommand{\fp}{(Ferone and Posteraro, 1992)}
\newcommand{\hlp}{(Hardy et al., 1964)}
\newcommand{\oc}{(Rockafellar, 1970)}
\newcommand{\tI}{(Talenti, 1976)}
\newcommand{\tII}{(Talenti, 1985)}

\newtheorem{lem}{Lemma}
\newtheorem{thm}{Theorem}

\theoremstyle{remark}

\newenvironment{sistema}%
{\left\lbrace\begin{array}{@{}l@{}}}%
{\end{array}\right.}
\title{Convex Symmetrization for Anisotropic Elliptic Equations with a lower order term}
\begin{document}
%1pt=0.376mm
%\maketitle
% \tableofcontents %indice delle sezioni
\vspace*{45 mm}
\begin{center}
{\bf \fontsize{13}{13}\selectfont 
Convex Symmetrization for Anisotropic Elliptic Equations with a lower order term}\\
Nota di Gianpaolo Piscitelli \footnote{Dipartimento di Matematica e Applicazioni \lq\lq R. Caccioppoli\rq\rq, Universit\`a di Napoli \lq\lq Federico II\rq\rq, Via Cintia, 80126 Napoli, Italy. $\mathtt{ gianpaolo.piscitelli@unina.it}$}\\
\vspace{11pt}
{\fontsize{10}{10}\selectfont Presentata dal socio Vincenzo Ferone\\
(Adunanza del 21 Novembre)}
\end{center}
\vspace{11pt}
{\fontsize{10}{10}\selectfont \textit{Keywords:} anisotropic elliptic equations, rearrangements, convex symmetrization}\\

{\fontsize{10}{10}\selectfont \noindent \textbf{Abstract} - We use \lq\lq generalized\rq\rq\ version of total variation, coarea formulas, isoperimetric inequalities to obtain sharp estimates for solutions (and for their gradients) to anisotropic elliptic equations with a lower order term, comparing them with the solutions to the convex symmetrized ones.}\\

{\fontsize{10}{10}\selectfont \noindent \textbf{Riassunto} - In questa nota si usano le versioni \lq\lq generalizzate \rq\rq della variazione totale, delle formule di coarea e delle diseguaglianze isoperimetriche al fine di ottenere stime ottimali per le soluzioni (e per i loro gradienti) di equazioni ellittiche anisotrope con un termine di ordine inferiore, confrontandole  con le soluzioni di quelle simmetrizzate convesse.}
\vspace{11pt}

\section{\fontsize{10}{10}\selectfont \bf 1 -  INTRODUCTION}
\thispagestyle{fancy}
%\markright{CONVEX SYMMETRIZATION FOR ANISOTROPIC ELLIPTIC EQUATIONS}
To reduce the complexity of a well defined class of problems, sometimes is possible to estimate the solutions by those of the corresponding symmetrized problem. \\

\medskip
In this way, we can deduce some information about the solutions to the generic problem using the solutions of the symmetrized one. So it is very significant to define an appropriate symmetrization.\\
By means of Schwarz (or spherical) symmetrization, it is possible to obtain comparison results for solutions to linear elliptic problems:
\begin{equation}
\label{semell}
-\text{div}( A(x)\cdot \nabla u) =f \quad\text{in}\ \Omega,\quad u \in H^1_0(\Omega),
\end{equation}
where $\nabla$ stands for the gradient operator, and $A(x)$ is a measurable function such that
\begin{equation}
\langle A(x)\cdot\xi,\ \xi\rangle \geq |\xi|^2, \quad \forall \xi\in\R^n.
\end{equation}
If $u$ is a solution to (\ref{semell}), then $u^*(x)\leq v(x)$, where $v$ solves 
\begin{equation}
-\Delta v =f^* \quad\text{in}\ \Omega^*,\quad v \in H^1_0(\Omega^*),
\end{equation}
where $f^*$ is the symmetrized function of $f$ and $\Omega^*$ is the ball centered in the origin such that $|\Omega|=|\Omega^*|$. For example, in \atl  and in \tII, we find comparison results to elliptic operators of general form, that is with first and zero order terms, with different constraints on the coefficients of lower order terms. Further results can be found in \abt,\atI, \bm, \fp, for linear cases and \bfm, \bm for non linear cases.
In \aftl, section 4, we find a comparison result for solutions to problem
\begin{equation}
\label{seanel}
 -\text{div} (a(x,u ,\nabla u))=f \quad \text{in} \ \Omega,\quad u \in H^1_0(\Omega)
\end{equation}
where
\begin{equation}
\label{anelco}
\langle a(x,\eta,\xi),\ \xi\rangle \geq H(\xi)^2 \quad \text{a.e.}\quad x \in\Omega ,\quad\eta\in\R ,\quad\xi\in\R^n,
\end{equation}
with $H$ homogeneous convex function. The authors, using convex symmetrization, estimate a solution of (\ref{seanel}) in terms of a function $v$ that solves
\begin{equation}
-\Delta v =f^* \quad\text{in}\ \Omega^*,\quad v \in H^1_0(\Omega^*).
\end{equation}
In the present paper we consider a lower order term $b(x, \nabla u)$ for (\ref{seanel}), that is
\begin{equation}
\label{aniell}
 -\text{div} (a(x,u,\nabla u))+b(x, \nabla u)=f \quad \text{in} \ \Omega,\quad u \in H^1_0(\Omega)
\end{equation}
where $a$ satisfies the ellipticity condition (\ref{anelco}) and on $b$ we assume that
\begin{equation}
|b(x,\xi)|\leq B(x) H(\xi)
\end{equation}
where $B(x)$ is an integrable function. Also in this case we use convex symmetrization, obtaining comparison results with solutions of the convexly symmetric problem
\begin{equation}
\label{symell}
\begin{sistema}
-\text{div} (H(\nabla v)\nabla H(\nabla v)) + \tilde{b}(H_0(x)) \langle \nabla H_0(x), \nabla H( \nabla v) \rangle H(\nabla v)=f^\star \  \text{in} \ \Omega^\star\\ v \in H_0^1(\Omega^\star),
\end{sistema}
\end{equation}
where $H_0$ is polar to $H$, $\tilde{b}$ is an auxiliary function related to $B$, $f^\star$ is the convex rearrangement of $f$ with respect to $H$ and $\Omega^\star$ is the set homothetic to
\begin{equation*}
K_0:=\lbrace x \in \R^n : H_0(x)\leq 1 \rbrace.
\end{equation*}
We obtain the following estimates:
\begin{align}
\label{estim1}
u^\star & \leq v\\
\label{estim2}
\int_{\Omega}H^q(\nabla u) & \leq\int_{\Omega^\star}H^q(\nabla v)
\end{align}
In the proof we use the generalized versions of total variation, coarea formulas and isoperimetric inequalities (\emph{see} \aftl, \tII). We derive some differential inequalities for the rearrangement $u^*$ of the solution $u$ using Schwarz and Hardy inequalities and the properties of homogeneity and convexity of the function $H$. Finally we consider the case where $\tilde{b}$ is essentially bounded by a constant $\beta$; we can compare solutions of (\ref{aniell}) with solutions to
\begin{equation}
\begin{sistema}
-\text{div} (H(\nabla v)\nabla H(\nabla v)) - \beta\langle \nabla H_0(x), \nabla H( \nabla v) \rangle H(\nabla v)=f^\star \  \text{in} \ \Omega^\star\\ v \in H_0^1(\Omega^\star)
\end{sistema}
\end{equation}
and we obtain the same estimates (\ref{estim1}) and (\ref{estim2}) of the preceding case.
We refer to \dgI and \dgII for similar results under different assumptions on $b(x,\xi)$.
\vspace{11pt}

\section{\fontsize{10}{10}\selectfont \bf 2 - PRELIMINARIES}
\subsection{\fontsize{10}{10}\selectfont \bf 2.1 - REARRANGEMENTS}
Let $\Omega$ be a measurable and not negligible subset of $n$-dimensional euclidean space $\mathbb{R}^n$, let $u$ be a measurable map from $\Omega$ into $\mathbb{R}$. We define (\emph{see} also \tI):\\
the \emph{distribution function} of $u$ as the map $\mu$ from $[0,\infty[$ to $[0, \infty[$  such that\\ \mbox{$\mu (t) := |\{x \in \Omega : |u(x)|>t \}|$};\\
the \emph{decreasing rearrangement} of $u$, denoted by $u^*$, as the  map from $[0,\infty[$ to $[0, \infty[$  such that \mbox{$u^* (s) :=\sup\{t>0 : \mu(t)>s \}$};\\ 
the \emph{sferically decreasing rearrangement} of $u$, denoted by $u^\#$, as the  map from $[0,\infty[$ to $[0, \infty[$  such that \mbox{$u^\#(s) :=\sup\{t>0 : \mu(t)>\omega_n|x|^n \}$}.\\ 
We denote by $\Omega^\#$ the ball centered in the origin such that $|\Omega^\#|=|\Omega| $.
\vspace{11pt}

\subsection{\fontsize{10}{10}\selectfont \bf 2.2 - GAUGE}
Let $H : \mathbb{R}^n \to [0,\infty[$ be a $C^1(\R^n\backslash\{ 0\})$ convex function satisfying the homogeneity property:
\begin{equation}
\label{hompro}
H(t\xi)=|t|H(\xi), \quad \forall \xi \in \mathbb{R}^n, \quad \forall t \in \mathbb{R}.
\end{equation}
Furthermore, assume that $H$ satisfies
\begin{equation}
\label{betabs}
\alpha | \xi |\leq H(\xi)\leq\beta|\xi|, \quad \forall \xi \in \mathbb{R}^n,
\end{equation}
for some positive constants $\alpha \leq \beta$. We also assume that
\begin{equation}
K=\{x \in \mathbb{R}^n : H(x)\leq 1\}
\end{equation}
has measure $|K|$ equal to the measure $\omega_n$ of the unit sphere in $\mathbb{R}^n$.  Because of ($\ref{hompro}$), this assumption is not restrictive. Sometimes we will say that $H$ is the gauge of $K$. If one defines the \emph{support function} of $K$ as:
\begin{equation}
H_0(x)= \sup_{\xi\in K} \langle x, \xi \rangle.
\end{equation} 
Clearly $H_0(x)$ itself is a gauge of the set:
\begin{equation}
K_0= \{x\in \mathbb{R}^n: H_0(x)\leq 1\},
\end{equation} 
we denote by $\kappa_n$ the measure of $K_0$.\\
Let us observe that $\nabla H_0(x)$ is, for a.e. $x$, a vector normal to $\partial K_0(x)$. Then the definition of $H$ and $H_0$ gives (\emph{see} \oc)
\begin{equation}
\label{defprH}
H(\nabla H_0(x))=\frac{\langle \nabla H_0(x),x\rangle}{H_0(x)}\quad \text{and} \quad H_0(\nabla H(x))=\frac{\langle \nabla H(x),x\rangle}{H(x)}.
\end{equation}
The homogeneity assumption (\ref{hompro}) implies, by Euler's Theorem, that
\begin{equation}
\label{eulhho}
H(\nabla H_0(x))=1 \quad \text{and} \quad H_0(\nabla H(x))=1.
\end{equation}
%Hence, if we consider $x^*=\nabla H(\nabla H_0(x))$, we have, by Euler's Theorem again and by (\ref{eulhho}a)
%\begin{equation}
%H_0(x^*)=\frac{\langle x^*,\nabla H_0(x)\rangle}{H(\nabla H_0(x))}=1,
%\end{equation}
%that implies
%\begin{equation}
%\label{invHo*}
%H(\nabla H_0(x))=\frac{\langle x^*,\nabla H_0(x)\rangle}{H_0(x^*)}=\langle x^*,\nabla H_0(x)\rangle.
%\end{equation}
%Now, equaling the last terms of (\ref{defprH}a) and (\ref{invHo*}) we have
It is useful to recall that by Euler Theorem we also have
\begin{equation}
\nabla H(\nabla H_0(x))=\frac{x}{H_0(x)}.
\end{equation}
We define the \emph{(decreasing) convex rearrangement} of $u$, denoted by $u^\star$, as the map such that \mbox{$u^\star(x)=u^*(\kappa_n(H_0(x))^n)$}.\\
We denote by $\Omega^\star$ the set homothetic to $K_0$ such that $|\Omega^\star|=|\Omega|$.
\vspace{11pt}

\subsection{\fontsize{10}{10}\selectfont \bf 2.3 - GENERALIZED TOTAL VARIATION, PERIMETER AND COAREA \\
 \mbox{\ \ \  FORMULA}}
It is possible to give the following definition of the total variation of a function $u \in BV(\Omega)$ with respect to a gauge function $H$ (\emph{see} \ab):
\begin{equation}
\int_\Omega |\nabla u|_H=\sup\left\{\int_\Omega u \  \text{div}   \varphi  \ dx  : \varphi \in C_0^1(\Omega; \mathbb{R}^n), H_0(\varphi)\leq 1  \right\}
\end{equation}
and the following \lq\lq generalized\rq\rq definition of perimeter of a set $E$ with respect to $H$:
\begin{equation}
P_H(E; \Omega)=\int_\Omega |\nabla \chi_E|_H=\sup\left\{\int_\Omega  \text{div}   \varphi  \ dx  : \varphi\in C_0^1(\Omega; \mathbb{R}^n), H_0(\varphi)\leq 1  \right\}
\end{equation}
These definitions yeld to the coarea formula
\begin{equation}
\label{fleris}
\int_\Omega |\nabla u|_H=\int_0^\infty P_H(\{u>s\}; \Omega )ds,
\end{equation}
and to the \lq\lq generalized\rq\rq\ isoperimetric inequality
\begin{equation}
\label{isoine}
P_H(E; \mathbb{R}^n )\geq n \kappa_n^{1/n}|E|^{1-\frac{1}{n}}.
\end{equation}
We finally observe that if $u\in W^{1,1}(\Omega)$ then
\begin{equation}
\int_\Omega |\nabla u|_H = \int_\Omega H(\nabla u),
\end{equation}
and it holds
\begin{equation}
\label{derfle}
-\frac{d}{dt}\int_{u>t} |\nabla u|_H dx =P_H(\{u>t\}; \Omega ).
\end{equation}
\vspace{11pt}

\subsection{\fontsize{10}{10}\selectfont \bf 2.4 - PRELIMINARY RESULTS}
In this section we give three Lemmas, that are basic for our treatment.
\begin{lem}
\label{mu1H2>}
If $u$ is any member of $H_0^1(\Omega)$, then 
\begin{equation}
\label{thlem1}
\frac{1}{n^2\kappa_n^\frac{2}{n}}\mu (t)^{\frac{2}{n}-2}\left[ -\mu ' (t)\right]\left[ -\frac{d}{dt} \int_{|u|>t} H^2(\nabla u) \right]\geq 1 
\end{equation}
for a.e. $t$ such that $0<t<\esssup |u|$.
\begin{proof}
For $h>0$, Schwarz inequality gives 
\begin{equation}
\frac{1}{h}\int_{t<|u| \leq t+h} H(\nabla u)\leq \frac{1}{h}\left(\int_{t<|u|\leq t+h}dx \right)^\frac{1}{2}\left(\int_{t<|u|\leq t+h} H^2(\nabla u) \right)^\frac{1}{2}
\end{equation}
and
\begin{equation}
\frac{1}{h}\int_{t<|u| \leq t+h} H(\nabla u)\leq \left(\frac{1}{h}(\mu(t)-\mu(t+h)\right)^\frac{1}{2}\left(\frac{1}{h}\int_{t<|u| \leq t+h} H^2(\nabla u) \right)^\frac{1}{2}.
\end{equation}
Therefore, as $h\rightarrow 0^+$, we obtain
\begin{equation}
\label{holder}
-\frac{d}{dt}\int_{|u|>t} H(\nabla u)\leq \left(-\mu '(t)\right)^\frac{1}{2}\left(-\frac{d}{dt}\int_{|u|>t} H^2(\nabla u) \right)^\frac{1}{2}.
\end{equation}
By (\ref{isoine}) and (\ref{derfle}), we have
\begin{equation}
n \kappa_n^{1/n}\mu(t)^{1-\frac{1}{n}}\leq \sqrt{-\mu '(t)}\left(-\frac{d}{dt}\int_{|u|>t} H^2(\nabla u) \right)^\frac{1}{2}.
\end{equation}
Then squaring and dividing by $n^2 \kappa_n^{2/n}\mu(t)^{2-\frac{2}{n}}$, we obtain (\ref{thlem1}).
\end{proof}
\end{lem}
\begin{lem}
\label{harlit}
\begin{equation}
\int_E|f| \leq \int_0^{|E|}f^*(s) \ ds
\end{equation}
for any measurable set $E$.
\end{lem}
This Lemma is a special case of a theorem by Hardy and Littlewood (\emph{see} \hlp, Theorem 378).
\begin{lem}
\label{gengro}
If $\varphi$ is bounded and 
\begin{equation}
\varphi(t)\leq\int_t^{+\infty}K(s)\varphi(s)\ ds + \psi (t) 
\end{equation}
for a.e. $t>0$, then
\begin{equation}
\varphi(t)\leq\int_t^{+\infty}\exp\left( \int_t^sK(r) \ dr\right) (-d\psi(s)) 
\end{equation}
for a.e. $t>0$. Here $K$ is any nonnegative integrable function, $\psi$ has bounded variation and vanishes at $+\infty$.
\end{lem}
Lemma \ref{gengro} is a generalization of Gronwall's lemma.
\vspace{11pt}

\section{\fontsize{10}{10}\selectfont \bf 3 - MAIN RESULT}
In this section we discuss the main result of the paper. It consists in showing that a solution to (\ref{aniell}) can be compared in term of a solution to (\ref{symell}), where the function $\tilde{b}$ is known as a pseudo rearrangement of $B(x)$. 
It can be defined as
\begin{equation}
\label{tildeb}
\tilde{b}\left(\left(\frac{s}{\kappa_n}\right)^\frac{1}{n}\right)=\left(\frac{d}{ds} \int_{|u|>u^*(s)}B^2(x) \right)^\frac{1}{2}, 
\end{equation}
We refer to \atII and \tII for further details.
\begin{thm}
\label{comthe}
Let $u \in H_0^1(\Omega)$ be a solution to the problem
\begin{equation}
\label{genpro}
\begin{sistema}
 -\text{\rm div} (a(x,u ,\nabla u))+b(x,\nabla u)=f \quad\text{in}\ \Omega\\
 u=0 \qquad\qquad\qquad\qquad\qquad\qquad\text{on}\ \partial\Omega
\end{sistema}
\end{equation}
where $a(x, \eta ,\xi)\equiv \{a_i(x,\eta,\xi)\}_{i=1,...,n}$ are Carath\'eodory functions satisfying
\begin{equation}
\label{ellhyp}
\langle a(x, \eta, \xi),\xi \rangle \geq (H(\xi))^2 \quad \text{a.e.} \quad  x\in\Omega, \quad  \eta\in\mathbb{R}, \quad  \xi\in\mathbb{R}^n.
\end{equation}
and $b(x,\xi)$ is such that:
\begin{equation}
\label{subhyp}
|b(x,\xi)|\leq B(x) H(\xi),
\end{equation}
where $B \in L^k(\Omega)$, with $k>n$. 
We assume further that $f \in L^\frac{2n}{n+2}(\Omega)$ if $n \geq 3$; $f \in L^p(\Omega)$, $p>1$, if $n=2$;\\
$H: \mathbb{R}^n \to [0,\infty[$ is a convex function satisfying \text{\rm (\ref{hompro})-(\ref{betabs})}.\\
Then
\begin{align}
\label{u*lqv*}
u^\star & \leq v\\
\label{ihqdul}
\int_{\Omega}H^q(\nabla u) & \leq\int_{\Omega^\star}H^q(\nabla v)
\end{align}
with $ 0 < q \leq 2 $, and 
\begin{equation}
\label{v(x)the}
v(x)=\int_{H_0(x)}^{\left(\frac{| \Omega |}{\kappa_n}\right)^{1/n}} \frac{1}{ t^{n-1}}dt \int_0^t \exp\left(\int_t^r\tilde{b}(r')dr'\right) f^*(\kappa_nr^n)r^{n-1}dr.
\end{equation}
where $\tilde{b}$ is defined as in \text{\rm (\ref{tildeb})}.
\end{thm}
\noindent
\textbf{Remark 1.} The function in (\ref{v(x)the}) is convexly symmetric, in the sense that $v(x)=v^\star(x)$. Indeed the function
\begin{equation}
v^*(s)=\int_s^{| \Omega|} \frac{1}{n^2\kappa_n^{2/n}} t^{\frac{2}{n}-2}dt \int_0^t \exp\left(\int_{\left(\frac{r}{\kappa_n}\right)^{1/n}}^{\left(\frac{t}{\kappa_n}\right)^{1/n}}\tilde{b}(r')dr'\right) f^*(r)dr
\end{equation}
is decresing and $v(x)=v^*(\kappa_n(H_0(x))^n)$. We observe that $v(x)$ is a solution in $H^1_0(\Omega^\star)$ to the problem
\begin{equation}
\label{cosypr}
\begin{sistema}
 -\text{div} (H(\nabla v)\nabla H(\nabla v))-\tilde{b}(H_0(x))\langle \nabla H_0(x),\nabla H( \nabla v) \rangle H(\nabla v)=f^\star \  \text{in} \ \Omega^\star\\
 v=0 \qquad\qquad\qquad\qquad\qquad\qquad\qquad\qquad\qquad\qquad\qquad\qquad\qquad  \text{on} \ \partial\Omega^\star.
\end{sistema}
\end{equation}
In fact, if we define $\rho=H_0(x)$ and we look for a solution such that $v(\rho)=v(H_0(x))$, we obtain
\begin{align}
\nabla v&=v'(\rho)\nabla H_0(x),\\
\label{subhdv}
H(\nabla v)&=-v'(\rho)H(\nabla H_0(x))=-v'(\rho),\\
\nabla H(\nabla v)&=\nabla H(v'(\rho)\nabla H_0(x))=\nabla H(\nabla H_0(x))=\frac{x}{H_0(x)}.
\end{align}
A direct computation gives
\begin{equation}
\label{symprovro}
\begin{split}
-\text{div}(H(\nabla v)\nabla H(\nabla v))- \tilde{b}(H_0(x)) & \langle \nabla H_0(x),  \nabla H(\nabla v)\rangle H(\nabla (v))\\
%&=-\text{div}\left(v'(\rho)\frac{x}{H_0(x)}\right)+\tilde{b}(H_0(x))\langle \nabla H_0(x),\frac{x}{H_0(x)}\rangle v'(\rho)\\
%&=-\text{div}\left(\frac{v'(\rho)}{\rho}x\right)+\tilde{b}(H_0(x))v'(\rho)\\
%&=-\left(\frac{v'(\rho)}{\rho}-\frac{v'(\rho)}{\rho^2}\right )\langle \nabla H_0(x), x\rangle- n\frac{v'(\rho)}{\rho}+\tilde{b}(H_0(x))v'(\rho)\\
&=-v''(\rho)-\frac{n-1}{\rho}v'(\rho)+\tilde{b}(H_0(x))v'(\rho).
\end{split}
\end{equation}
Using (\ref{v(x)the}), we can write:
\begin{equation}
\label{comhdv}
v(\rho)=\int_\rho^{\left(\frac{| \Omega |}{\kappa_n}\right)^{1/n}} \frac{1}{ t^{n-1}}dt \int_0^t \exp\left(\int_\rho^tg(r')dr'\right) f^*(\kappa_nr^n)r^{n-1}dr
\end{equation}
%we can calculate the first derivative
%\begin{equation}
%v'(\rho)=-\frac{1}{\rho^{n-1}}I,
%\end{equation}
%and the second derivative of $v$:
%\begin{equation}
%v''(\rho)=\frac{n-1}{\rho^n}I-\frac{B}{\rho^{n-1}}I- f^*(\kappa_n\rho^n)=\frac{n-1}{\rho^n}I-\frac{B}{\rho^{n-1}}I- f^\star(x),
%\end{equation}
%where
%\begin{equation}
%I(\rho)=\int_0^\rho \exp\left(B(\rho-r)\right) f^*(\kappa_nr^n)r^{n-1}dr.
%\end{equation}
%Inserting (\ref{firder}) and (\ref{secder}) in the last term of (\ref{symprovro}) we obtain
%\begin{equation}
%-\frac{n-1}{\rho^n}I+\frac{B}{\rho^{n-1}}I+ f^*(\kappa_n\rho^n)+\frac{n-1}{\rho^n}I-\frac{B}{\rho^{n-1}}I= f^\star(x).
%\end{equation}
and we have:
\begin{equation}
\label{compuv}
-v''(\rho)-\frac{n-1}{\rho}v'(\rho)+\tilde{b}(H_0(x))v'(\rho)=f^\star(\rho).
\end{equation}
Collecting (\ref{comhdv}) and (\ref{compuv}) we obtain that the function in (\ref{v(x)the}) solves (\ref{cosypr}).
%\begin{equation}
%v^\star(x)=\int_{H_0(x)}^{\left(\frac{| \Omega |}{\kappa_n}\right)^{1/n}} \frac{1}{ t^{n-1}}dt \int_0^t \exp\left(\int_t^r\tilde{b}(r')dr'\right) f^*(\kappa_nr^n)r^{n-1}dr.
%\end{equation}
%An integration by the substitution $\xi=\kappa_n r^n$, gives
%\begin{equation}
%v^\star(x)=\int_{H_0(x)}^{\left(\frac{| \Omega |}{\kappa_n}\right)^{1/n}} \frac{1}{n\kappa_n} t^{1-n}dt \int_0^{\kappa_n t^n} \exp\left(\int_{\left(\frac{\xi}{\kappa_n}\right)^{1/n}}^t\tilde{b}(r')dr'\right) f^*(\xi)d\xi,
%\end{equation}
%therefore, by the substitution $\tau=\kappa_n t^n$, we have
%\begin{equation}
%v^*(\kappa_nH_0^n(x))=\int_{\kappa_nH_0^n(x)}^{| \Omega |} \frac{1}{n^2\kappa_n^{2/n}} \tau^{\frac{2}{n}-2}d\tau \int_0^\tau \exp\left(\int_{\left(\frac{\xi}{\kappa_n}\right)^{1/n}}^{\left(\frac{\tau}{\kappa_n}\right)^{1/n}}\tilde{b}(r')dr'\right) f^*(\xi)d\xi.
%\end{equation}
%It we put $s=\kappa_n H_0(x)^n$, we gain
%\begin{equation}
%v^*(s)=\int_s^{| \Omega |} \frac{1}{n^2\kappa_n^{2/n}} \tau^{\frac{2}{n}-2}d\tau \int_0^\tau \exp\left(\int_{\left(\frac{\xi}{\kappa_n}\right)^{1/n}}^{\left(\frac{\tau}{\kappa_n}\right)^{1/n}}\tilde{b}(r')dr'\right) f^*(\xi)d\xi.
%\end{equation}
\\
\textbf{Remark 2.} We can compute $\int_{\Omega^\star}H^q(\nabla v)$. By (\ref{subhdv}) we have
\begin{equation}
[H(\nabla v(x))]^q=[v'(\rho)]^q=\left[-\frac{1}{\rho^{n-1}}\int_0^\rho \exp\left(\int^\rho_r\tilde{b}(r')dr'\right) f^*(\kappa_n r^n)r^{n-1}dr\right]^q
\end{equation}
where $\rho = H_0(x)$. An integration by the substitution $s=\kappa_n r^n$ gives
\begin{equation}
[H(\nabla v(x))]^q=\left[-\frac{1}{n\kappa_n\rho^{n-1}}\int_0^{\kappa_n\rho^n} \exp\left(\int^\rho_{\left(\frac{s}{\kappa_n}\right)^{1/n}}\tilde{b}(r')dr'\right) f^*(s)ds\right]^q,
\end{equation}
therefore, by an integration on $\Omega^\star$, we have
\begin{equation}
\begin{split}
\int_{\Omega^\star} & [H(\nabla v(x))]^q\\
& = \int_0^{|\Omega|}\left[-\frac{1}{n\kappa_n\rho^{n-1}}\int_0^{\kappa_n\rho^n} \exp\left(\int^\rho_{\left(\frac{s}{\kappa_n}\right)^{1/n}}\tilde{b}(r')dr'\right)f^*(s)ds\right]^q d\rho.
\end{split}
\end{equation}
Hence, by the sustitution $\tau=\kappa_n\rho^n$, we have
\begin{equation}
\begin{split}
\int_{\Omega^\star} & [H(\nabla v(x))]^q\\
&=\int_0^{|\Omega|}\left[-\frac{1}{n\kappa_n^{1/n}}\tau^{\frac{1}{n}-1}\int_0^\tau \exp\left(\int^{\left(\frac{\tau}{\kappa_n}\right)^{1/n}}_{\left(\frac{s}{\kappa_n}\right)^{1/n}}\tilde{b}(r')dr'\right) f^*(s)ds\right]^q d\tau.
\end{split}
\end{equation}
\begin{thm}
\label{thecon}
Let $u \in H_0^1(\Omega)$ be a solution to problem \text{\rm (\ref{genpro})} under the assumption \text{\rm (\ref{ellhyp})}. Furthermore we suppose that \text{\rm (\ref{subhyp})} holds with
\begin{equation}
||B||_{L^\infty(\Omega)}=\beta\leq \infty;
\end{equation}
$f \in L^\frac{2n}{n+2}(\Omega)$ if $n \geq 3$; $f \in L^p(\Omega)$, $p>1$, if $n=2$; $H: \mathbb{R}^n \to [0,\infty[$ is a convex function satisfying \text{\rm (\ref{hompro})-(\ref{betabs})}.\\
Then \text{\rm (\ref{u*lqv*})} and \text{\rm (\ref{ihqdul})} holds with
\begin{equation}
\label{vbetac}
v(x)=\int_{H_0(x)}^{\left(\frac{| \Omega |}{\kappa_n}\right)^{1/n}} \frac{1}{ t^{n-1}}dt \int_0^t e^{\beta(r-t)} f^*(\kappa_nr^n)r^{n-1}dr.
\end{equation}
\end{thm}
\noindent
\textbf{Remark 3.} The function $v(x)$ in (\ref{vbetac}) is a solution in $H_0^1(\Omega^\star)$ to the problem
\begin{equation}
\begin{sistema}
 -\text{\rm div} (H(\nabla v)\nabla H(\nabla v))-\beta\langle \nabla H_0(x),\nabla H( \nabla v) \rangle H(\nabla v)=f^\star \  \text{in} \ \Omega^\star\\
 v=0 \qquad\qquad\qquad\qquad\qquad\qquad\qquad\qquad\qquad\qquad\qquad\quad \text{on} \ \partial\Omega^\star.
\end{sistema}
\end{equation}
The proof of Theorem \ref{thecon} is similar to that of Theorem \ref{comthe} and it can be obtained from it considering the function $B(x)$ as a constant.
\vspace{11pt}

\section{\fontsize{10}{10}\selectfont \bf 4 - PROOF OF THEOREM \ref{comthe}}
Let us start by proving a preliminary result about the function $\tilde{b}$ (\emph{see} \tII).
\begin{lem}
\label{lechtb}
If $\tilde{b}$ is defined by \text{\rm (\ref{tildeb})}, then 
\begin{equation}
\label{chatib}
\left(-\frac{d}{dt} \int_{|u|>t}B^2(x) \right)^\frac{1}{2}=\sqrt{-\mu ' (t)} \ \tilde{b}\left(\left(\frac{\mu(t)}{\kappa_n}\right)^\frac{1}{n}\right)
\end{equation}
and
\begin{equation}
\label{thlebt}
-\frac{d}{dt}\int_{|u|>t}B(x)H(\nabla u)\leq \left(-\frac{d}{dt}\int_0^{\left(\frac{\mu (t)}{\kappa_n}\right)^\frac{1}{n}}\tilde{b}(r)dr\right)\left(-\frac{d}{dt}\int_{|u|>t}H^2(\nabla u)\right)
\end{equation}
for almost every $t \in [0,\esssup_\Omega |u|] $.
\end{lem}
\begin{proof}
Let $p(t)$ and $q(s)$  be the integrals of $B(x)$ over $\{|u|>t\}$ and $\{|u|>u^*(s)\}$ respectively, hence $p'(t)=q'(\mu(t))\mu ' (t)$ for almost every $t\in[0,\esssup_\Omega u]$. So equality (\ref{chatib}) is proved.\\
By H\"older inequality, we have
\begin{equation}
-\frac{d}{dt}\int_{|u|>t}B(x)H(\nabla u)\leq \left(-\frac{d}{dt}\int_{|u|>t}B(x)\right)^\frac{1}{2}\left(-\frac{d}{dt}\int_{|u|>t}H^2(\nabla u)\right)^\frac{1}{2},
\end{equation}
by (\ref{chatib}) we obtain
\begin{equation}
-\frac{d}{dt}\int_{|u|>t}B(x)H(\nabla u)\leq\sqrt{-\mu ' (t)} \ \tilde{b}\left(\left(\frac{\mu(t)}{\kappa_n}\right)^\frac{1}{n}\right)\left(-\frac{d}{dt}\int_{|u|>t}H^2(\nabla u)\right)^\frac{1}{2},
\end{equation}
hence, by Lemma \ref{mu1H2>} ,
\begin{equation}
\begin{split}
-\frac{d}{dt}\int_{|u|>t} & B(x)H(\nabla u)\\
& \leq -\mu'(t)\frac{\mu(t)^{\frac{1}{n}-1}}{n\kappa_n^{1/n}} \ \tilde{b}\left(\left(\frac{\mu(t)}{\kappa_n}\right)^\frac{1}{n}\right)\left(-\frac{d}{dt}\int_{|u|>t}H^2(\nabla u)\right),
\end{split}
\end{equation}
that is equal to the right-hand side of (\ref{thlebt}).
\end{proof}
\noindent
\textbf{Proof of Theorem \ref{comthe}.}
Suppose $u$ is a weak solution of problem (\ref{genpro}), then
\begin{equation}
\label{weasol}
\int_\Omega\langle a(x,u,\nabla u),\nabla \varphi\rangle +\int_\Omega b(x,\nabla u)\varphi=\int_\Omega f\varphi,\qquad  \forall \varphi\in H_0^1(\Omega).
\end{equation}
For $h>0$, $t>0$, let $\varphi$ be the following test function
\begin{equation}
\varphi_h(x)=
\begin{sistema}
  h,\qquad\quad \text{if} \  |u|>t+h\\
 |u|-t, \quad \text{if} \ t<|u|\leq t+h\\
 0, \qquad\quad \text{if} \ |u|\leq t,
\end{sistema}
\end{equation}
then
\begin{equation}
\nabla _i\varphi_h(x)=
\begin{sistema}
0,\qquad \text{if} \  |u|>t+h\\
\nabla _iu, \quad \text{if} \  t<|u|\leq t+h\\
0, \qquad \text{if} \ |u|\leq t.
\end{sistema}
\end{equation}
Inserting this test function in (\ref{weasol}), we have
\begin{equation}
\begin{split}
\int_{t<|u| \leq t+h} & \langle  a(x,u,\nabla u) ,\nabla  u\rangle + \int_{|u|>t+h} b(x,\nabla u)h\\
&=\int_{|u|>t+h} fh +\int_{t<|u| \leq t+h}(f-b(x,\nabla u))(|u|-t)\sgn u.
\end{split}
\end{equation}
The last term is smaller than $ \int_{t<|u| \leq t+h}(f- b(x,\nabla u))(|u|-t)$ and, by hypothesis (\ref{ellhyp}) and (\ref{subhyp}), we have
\begin{equation}
\label{symwea}
\begin{split}
\int_{t<|u|\leq t+h}H^2(\nabla u) -h \int_{|u|>t+h} B(x) & H(\nabla u)\leq \int_{|u|>t+h} fh \\
& +\int_{t<|u|\leq t+h}(f- b(x,\nabla u))(|u|-t).
\end{split}
\end{equation}
%Let us first observe that, clearly,
%\begin{equation}
%\lim_{h\rightarrow 0}\frac{1}{h}\int_{t<|u| \leq t+h}... \ dx %=\lim_{h\rightarrow 0}\frac{1}{h}\left(\int_{|u|>t}... \ dx - \int_{|u|>t+h}... \ dx\right) =-\lim_{h\rightarrow 0}\frac{1}{h}\left(\int_{|u|>t+h}... \ dx - \int_{|u|>t}... \ dx\right)=-\frac{d}{dt}\int_{|u|>t}... \ dx
%\end{equation}
Dividing each term by $h$, as $h\rightarrow 0^+$, (\ref{symwea}) becomes
\begin{equation}
-\frac{d}{dt}\int_{|u|>t}H^2(\nabla u) - \int_{|u|>t} B(x)H(\nabla u)\leq\int_{|u|>t} f,
\end{equation}
and, by Lemma \ref{harlit},
\begin{equation}
\label{weagro}
-\frac{d}{dt}\int_{|u|>t}H^2(\nabla u) - \int_{|u|>t} B(x)H(\nabla u)\leq\int_0^{\mu(t)} f^*(s)ds.
\end{equation}
Now, we can write
\begin{equation}
\int_{|u|>t}B(x)H(\nabla u)= \int_t^{+\infty}\left(-\frac{d}{ds}\int_{|u|>s}B(x) H(\nabla u)\right)ds
\end{equation}
and hence, by Lemma \ref{lechtb}, we have
\begin{equation}
\label{flehol}
\begin{split}
\int_{|u|>t} & B(x)H(\nabla u) \\
& \leq\int_t^{+\infty}\left(-\frac{d}{dt}\int_0^{\left(\frac{\mu(t)}{\kappa_n}\right)^\frac{1}{n}}\tilde{b}(r)dr\right)\left(-\frac{d}{ds}\int_{|u|>s} H^2(\nabla u)\right) ds.
\end{split}
\end{equation}
Inserting (\ref{flehol}) in (\ref{weagro}) we obtain
\begin{equation}
\begin{split}
&- \frac{d}{dt}\int_{|u|>t}  H^2(\nabla u) \\
& \leq\int_t^{+\infty}\left(-\frac{d}{dt}\int_0^{\left(\frac{\mu(t)}{\kappa_n}\right)^\frac{1}{n}}\tilde{b}(r)dr\right)  \left(-\frac{d}{ds}\int_{|u|>s} H^2(\nabla u)\right) ds +\int_0^{\mu(t)} f^*(s)ds.
\end{split}
\end{equation}
Now we can use Lemma \ref{gengro} with $\varphi (t)=-\frac{d}{dt}\int_{|u|>t}H^2(\nabla u)$. We have
\begin{equation}
-\frac{d}{dt}\int_{|u|>t}H^2(\nabla u) \leq\int_t^{+\infty}\exp\left(\int_t^s-\frac{d}{dr}\int_0^{\left(\frac{\mu(r)}{\kappa_n}\right)^\frac{1}{n}}\tilde{b}(r')dr'\right)\left[-d\psi( s)ds\right],
\end{equation}
where $\psi(s)=\int_0^{\mu(s)} f^*(\xi)d\xi$.\\
Using the substitution $\rho=\mu(s)$ and $\sigma=\mu(r)$, we obtain
%\begin{equation}
%-\frac{d}{dt}\int_{|u|>t}H^2(\nabla u) \leq-\int_{\mu(t)}^0\exp\left(-\int_{\mu(t)}^\rho\frac{B}{n\kappa_n^{1/n}}z^{\frac{1}{n}-1}dz\right)f^*(\rho)d\rho,
%\end{equation}
%hence we have
\begin{equation}
\label{wd2f13}
-\frac{d}{dt}\int_{|u|>t}H^2(\nabla u) \leq\int_0^{\mu(t)}\exp\left(\int_{\left(\frac{\sigma}{\kappa_n}\right)^\frac{1}{n}}^{\left(\frac{\mu(t)}{\kappa_n}\right)^\frac{1}{n}}\tilde{b}(\rho)d\rho\right)f^*(\sigma)d\sigma.
\end{equation}
Inequality (\ref{wd2f13}) and Lemma \ref{mu1H2>} give 
\begin{equation}
1\leq\frac{1}{n^2\kappa_n^{2/n}}\mu(t)^{\frac{2}{n}-2}(-\mu'(t))\int_0^{\mu(t)}\exp\left(\int_{\left(\frac{\sigma}{\kappa_n}\right)^\frac{1}{n}}^{\left(\frac{\mu(t)}{\kappa_n}\right)^\frac{1}{n}}\tilde{b}(\rho)d\rho\right)f^*(\sigma)d\sigma.
\end{equation}
for a.e. $t \in [0, \esssup |u|]$, then integration of both sides with respect to $t$ over the interval $[0,u^*(s)]$ yields
%\begin{equation}
%u^*(s)\leq\int_0^{\mu^*(s)}\frac{1}{n^2\kappa_n^{2/n}}\mu(t)^{\frac{2}{n}-2}(-\mu'(t))\int_0^{\mu(t)}\exp\left(\int_s^{\mu(t)}\frac{B}{n\kappa_n^{1/n}}r^{\frac{1}{n}-1}dr\right)f^*(s)ds
%\end{equation}
%By an integration by substitution, we can obtain that
\begin{equation}
\label{u*<=v*}
u^*(s)\leq\int_s^{|\Omega|}dt\frac{1}{n^2\kappa_n^{2/n}}t^{\frac{2}{n}-2}\int_0^t\exp\left(\int_{\left(\frac{\sigma}{\kappa_n}\right)^\frac{1}{n}}^{\left(\frac{t}{\kappa_n}\right)^\frac{1}{n}}\tilde{b}(\rho)d\rho\right)f^*(\sigma)d\sigma
\end{equation}
From formula (\ref{v(x)the}), we learn that $v^*(s)$ is the right-hand side of (\ref{u*<=v*}), so (\ref{u*lqv*}) is satisfied.\\
In order to prove (\ref{ihqdul}), we observe that  H\"older inequality gives
\begin{equation}
\frac{1}{h}\int_{t<|u|\leq t+h}H^q(\nabla u) \leq \left(\frac{1}{h} \int_{t<|u|\leq t+h} \ dx \right)^{1-\frac{q}{2}}\left(\frac{1}{h}\int_{t<|u|\leq t+h}H^2(\nabla u)\right)^\frac{q}{2}
\end{equation}
and hence, for $t\rightarrow 0^+$,
\begin{equation}
\label{limhol}
-\frac{d}{dt}\int_{|u|>t}H^q(\nabla u) \leq \left(-\mu '(t) \right)^{1-\frac{q}{2}}\left(-\frac{d}{dt}\int_{|u|> t}H^2(\nabla u)\right)^\frac{q}{2},
\end{equation}
provided that $ 0 < q \leq 2$.
Lemma \ref{mu1H2>}  gives
\begin{equation}
\label{le1HDu}
\left[-\frac{d}{dt}\int_{|u|>t}H^2(\nabla u)\right]^\frac{1}{2}%=\frac{\left[-\frac{d}{dt}\int_{|u|>t}H^2(\nabla u)\right]}{\left[-\frac{d}{dt}\int_{|u|>t}H^2(\nabla u)\right]^\frac{1}{2}}
\leq\frac{1}{n\kappa_n^{1/n}}\mu(t)^{\frac{1}{n}-1}(-\mu ' (t))^\frac{1}{2}\left[-\frac{d}{dt}\int_{|u|>t}H^2(\nabla u)\right],
\end{equation}
hence by inequality (\ref{wd2f13})
\begin{equation}
\label{le1405}
\begin{split}
&\left[-\frac{d}{dt}\int_{|u|>t} H^2(\nabla u)\right]^\frac{1}{2}\\
& \leq\frac{1}{n\kappa_n^{1/n}}\mu(t)^{\frac{1}{n}-1}(-\mu ' (t))^\frac{1}{2}\int_0^{\mu(t)}\exp\left(\int_{\left(\frac{\sigma}{\kappa_n}\right)^\frac{1}{n}}^{\left(\frac{\mu(t)}{\kappa_n}\right)^\frac{1}{n}}\tilde{b}(\rho)d\rho\right)f^*(\sigma)d\sigma.
\end{split}
\end{equation}
Coupling (\ref{le1405}) with (\ref{limhol})
\begin{equation}
\begin{split}
-\frac{d}{dt}  & \int_{|u|>t}  H^q(\nabla u)\\
& \leq (-\mu ' (t))^{1-q/2}\left[\left(-\frac{d}{dt}\int_{|u|>t}H^2(\nabla u)\right)^\frac{1}{2}\right]^q\\
&\leq (-\mu ' (t))^{1-q/2}\\
&\quad\left[ \frac{1}{n\kappa_n^{(1/n)}}\mu(t)^{\frac{1}{n}-1}(-\mu ' (t))^\frac{1}{2}\int_0^{\mu(t)}\exp\left(\int_{\left(\frac{\sigma}{\kappa_n}\right)^\frac{1}{n}}^{\left(\frac{\mu(t)}{\kappa_n}\right)^\frac{1}{n}}\tilde{b}(\rho)d\rho\right)f^*(\sigma)d\sigma \right]^q.
\end{split}
\end{equation}
Consequently
\begin{equation}
\begin{split}
\int_\Omega & H^q(\nabla u)\\
&\leq \int_0^{|\Omega |}-\mu '(t)\\
&\qquad\quad\left[ \frac{1}{n\kappa_n^{(1/n)}}\mu(t)^{\frac{1}{n}-1}\int_0^{\mu(t)}\exp\left(\int_{\left(\frac{\sigma}{\kappa_n}\right)^\frac{1}{n}}^{\left(\frac{\mu(t)}{\kappa_n}\right)^\frac{1}{n}}\tilde{b}(\rho)d\rho\right)f^*(\sigma)d\sigma \right]^q dt,
\end{split}
\end{equation}
and hence, by the substitution $\tau=\mu(t)$,
\begin{equation}
\begin{split}
\int_\Omega  H^q( &\nabla u) \\
& \leq \int_0^{|\Omega |}\left[ \frac{1}{n\kappa_n^{(1/n)}}\tau^{\frac{1}{n}-1}\int_0^{\tau}\exp\left(\int_{\left(\frac{\sigma}{\kappa_n}\right)^\frac{1}{n}}^{\left(\frac{\tau}{\kappa_n}\right)^\frac{1}{n}}\tilde{b}(\rho)d\rho\right)f^*(\sigma)d\sigma \right]^q  d\tau\\
&\qquad\qquad\qquad\qquad\qquad\qquad\qquad\qquad\qquad\qquad\qquad =\int_{\Omega^\star}H^q(\nabla v ),
\end{split}
\end{equation}
so the theorem is proved.\\
%In this section we give the computation of the sperically rearrangement of a solution $v$ of the convexly symmetric problem $v^*$, and of $\int_{G^\star} H^q(\nabla v)$.

%\begin{center}{\textbf{Acknowledgments}}
%\end{center}
%The author wish to thank $\cdots$ \\ \\
%\vskip 0.4 true cm
\vspace{11pt}

\section{\fontsize{10}{10}\selectfont \bf 5 - REFERENCES}
%\begin{thebibliography}{20}
%\bibitem{AB} 
\indent Amar M. and Bellettini G. (1994) {\em A notion of total variation depending on a metric with discontinuous coefficients}.  Annales de l'Institut Henry Poincar\'e. Analyse Nonlineaire. {\bf 11}, pp. 91-133.\\
%\bibitem{ABT}
\indent Alvino A., Buonocore P. and Trombetti G. (1990) {\em On Dirichlet problem for second order elliptic equations}. Nonlinear Analysis: Theory, Methods \& Applications. {\bf 14}, no. 7, pp. 559-570. \\
%\bibitem{AFTL}
\indent Alvino A., Ferone V., Trombetti G. and Lions P.L. (1997) {\em Convex symmetrization and applications}. Annales de l'Institut Henry Poincar\'e (C) Non Linear Analysis. {\bf 14}, no. 2, pp. 275-293 \\
%\bibitem{AT}
\indent Alvino A. and Trombetti G. (1979) {\em Equazioni ellittiche con termini di ordine inferiore e riordinamenti}.  Atti della Accademia Nazionale dei Lincei. Rendiconti della Classe di Scienze Fisiche, Matematiche e Naturali (8). {\bf 66}, no. 3, pp. 194-200. \\
%\bibitem{AT2}
\indent Alvino A. and Trombetti G. (1978) {\em Sulle migliori costanti di maggiorazione per una classe di equazioni ellittiche}. Ricerche di Matematica. {\bf27}, no. 2, pp. 413-428\\
%\bibitem{ATL}
\indent Alvino A., Trombetti G. and Lions P.L. (1990) {\em Comparison results for elliptic and parabolic equations via Schwarz symmetrization}. Annales de l'Institut Henry Poincar\'e. Analyse Nonlineaire. {\bf 7}, no. 2, pp. 37-65 \\
%\bibitem{BFM}
\indent Betta M.F., Ferone V. and Mercaldo A. (1994) {\em Regularity for solutions of nonlinear elliptic equations}. Bulletin des Sciences Mathematiques. {\bf 118}, no. 6, pp. 539-567\\
%\bibitem{BM}
\indent Betta M.F. and Mercaldo A. (1991) {\em Existence and regularity results for a nonlinear elliptic equation}.  Rendiconti di Matematica e delle sue Applicazioni. {\bf 11}, no. 4, pp. 737-759\\
%\bibitem{DG}
\indent Della Pietra F. and Gavitone N. {\em Sharp estimates and existence for anisotropic elliptic problems with general growth in the gradient}. arXiv:1402.3086\\
%\bibitem{DG2}
\indent Della Pietra F. and Gavitone N. (2013) {\em Anisotropic elliptic equations with general growth in the gradient and Hardy-type potentials}. Journal of Differential Equations. {\bf 255}, pp. 3788-3810\\
%\bibitem{FFV}
\indent Ferone A., Ferone V. and Volpicelli R. (1997) {\em Moser-type inequalities for solutions of linear elliptic equations with lower-order terms}. Differential Integral Equations. {\bf 10}, no. 6, pp. 1031-1048.\\
%\bibitem{FP}
\indent Ferone V. and Posteraro M. (1992) {\em Symmetrization results for elliptic equations with lower-order terms}. Atti del Seminario Matematico e Fisico dell'Università di Modena. {\bf 40}, pp. 47-61.\\ 
%\bibitem{HLP}
\indent Hardy G.H., Littlewood J.E. and P\'olya G. (1964) {\em Inequalities}. Cambridge University Press, Cambridge, United Kingdom. \\
%\bibitem{R}
\indent Rockafellar R.T. (1970) {\em Convex Analysis}. Princeton University Press, Princeton, United States of America.\\
%\bibitem{T1}
\indent Talenti G. (1976) {\em Elliptic equations and rearrangements}. Annali della Scuola Normale Superiore di Pisa. Classe di Scienze(4). {\bf 3}, no. 4, pp. 697-718.\\
%\bibitem{T2}
\indent Talenti G. (1985) {\em Linear Elliptic P.D.E.'s: Level Sets, Rearrangements and a priori Estimates of Solutions}. Bollettino dell'Unione Matematica Italiana B (6), {\bf 4-B}, pp.917-949.\\
%\end{thebibliography}

\end{document}